\documentclass[reqno, 12pt]{amsart}
\pdfoutput=1
\makeatletter
\let\origsection=\section \def\section{\@ifstar{\origsection*}{\mysection}} 
\def\mysection{\@startsection{section}{1}\z@{.7\linespacing\@plus\linespacing}{.5\linespacing}{\normalfont\scshape\centering\S}}
\makeatother        

\usepackage{amsmath,amssymb,amsthm}
\usepackage{mathrsfs}
\usepackage{mathabx}\changenotsign
\usepackage{dsfont}
 
\usepackage{xcolor}
\usepackage[backref]{hyperref}
\hypersetup{
    colorlinks,
    linkcolor={red!60!black},
    citecolor={green!60!black},
    urlcolor={blue!60!black}
}

\usepackage{graphicx}

\usepackage{verbatim}

\usepackage[open,openlevel=2,atend]{bookmark}

\usepackage[abbrev,msc-links,backrefs]{amsrefs} 
\usepackage{doi}

\renewcommand{\PrintDOI}[1]{\doi{#1}}

\usepackage[T1]{fontenc}
\usepackage{lmodern}
\usepackage[babel]{microtype}
\usepackage[english]{babel}

\usepackage{geometry}
\numberwithin{equation}{section}
\numberwithin{figure}{section}

\usepackage{enumitem}

\let\polishlcross=\l
\def\l{\ifmmode\ell\else\polishlcross\fi}

\def\paragraph#1{%
  \noindent\textbf{#1.}\enspace}

\let\setminus=\smallsetminus

\makeatletter
\def\moverlay{\mathpalette\mov@rlay}
\def\mov@rlay#1#2{\leavevmode\vtop{   \baselineskip\z@skip \lineskiplimit-\maxdimen
   \ialign{\hfil$\m@th#1##$\hfil\cr#2\crcr}}}
\newcommand{\charfusion}[3][\mathord]{
    #1{\ifx#1\mathop\vphantom{#2}\fi
        \mathpalette\mov@rlay{#2\cr#3}
      }
    \ifx#1\mathop\expandafter\displaylimits\fi}
\makeatother

\DeclareFontFamily{U}  {MnSymbolC}{}
\DeclareSymbolFont{MnSyC}         {U}  {MnSymbolC}{m}{n}
\DeclareFontShape{U}{MnSymbolC}{m}{n}{
    <-6>  MnSymbolC5
   <6-7>  MnSymbolC6
   <7-8>  MnSymbolC7
   <8-9>  MnSymbolC8
   <9-10> MnSymbolC9
  <10-12> MnSymbolC10
  <12->   MnSymbolC12}{}
\DeclareMathSymbol{\powerset}{\mathord}{MnSyC}{180}

\let\epsilon=\varepsilon

\let\rho=\varrho
\let\theta=\vartheta

\theoremstyle{plain}
\newtheorem{thm}{Theorem}[section]
\newtheorem{thm-intro}{Theorem}[]
\newtheorem{theorem}[thm]{Theorem}
\newtheorem{corollary}[thm]{Corollary}

\newtheorem{proposition}[thm]{Proposition}

\newtheorem{lemma}[thm]{Lemma}

\theoremstyle{definition}
\newtheorem{remark}[thm]{Remark}
\newtheorem{definition}[thm]{Definition}
\newtheorem{example}[thm]{Example}

\usepackage{accents}

\let\phi=\varphi

\def\lowfwd #1#2#3{{\mathop{\kern0pt #1}\limits^{\kern#2pt\raise.#3ex
\vbox to 0pt{\hbox{$\scriptscriptstyle\rightarrow$}\vss}}}}
\def\lowbkwd #1#2#3{{\mathop{\kern0pt #1}\limits^{\kern#2pt\raise.#3ex
\vbox to 0pt{\hbox{$\scriptscriptstyle\leftarrow$}\vss}}}}
\def\fwd #1#2{{\lowfwd{#1}{#2}{15}}}
\def\ve{\kern-1pt\lowfwd e{1.5}2\kern-1pt}
\def\ev{\kern-1pt\lowbkwd e{1.5}2\kern-1pt}
\def\vp{\lowfwd p{1.5}2}

\def\vr{\lowfwd r{1.5}2}
\def\rv{\lowbkwd r02}

\def\vrdash{{\mathop{\kern0pt r\lower.5pt\hbox{${}
     \scriptstyle'$}}\limits^{\kern0pt\raise.02ex
     \vbox to 0pt{\hbox{$\scriptscriptstyle\rightarrow$}\vss}}}}
\def\rvdash{{\mathop{\kern0pt r\lower.5pt\hbox{${}
     \scriptstyle'$}}\limits^{\kern0pt\raise.02ex
     \vbox to 0pt{\hbox{$\scriptscriptstyle\leftarrow$}\vss}}}}

\def\vs{\lowfwd s{1.5}1}
\def\sv{\lowbkwd s{1.5}1}

\def\vsidash{{\mathop{\kern0pt s_i\kern-3.5pt\lower.3pt\hbox{${}
     \scriptstyle'$}}\limits^{\kern0pt\raise.02ex
     \vbox to 0pt{\hbox{$\scriptscriptstyle\rightarrow$}\vss}}}}
\def\vS{{\hskip-1pt{\fwd S3}\hskip-1pt}} 
\def\vSr{{\vec S}_{\raise.1ex\vbox to 0pt{\vss\hbox{$\scriptstyle\ge\vr$}}}}
\def\vSdash{{\mathop{\kern0pt S\lower-1pt\hbox{${}
     \scriptstyle'$}}\limits^{\kern2pt\raise.1ex
     \vbox to 0pt{\hbox{$\scriptscriptstyle\rightarrow$}\vss}}}}
\def\vsdash{{\mathop{\kern0pt s\lower.5pt\hbox{${}
     \scriptstyle'$}}\limits^{\kern0pt\raise.02ex
     \vbox to 0pt{\hbox{$\scriptscriptstyle\rightarrow$}\vss}}}}
     
\def\svdash{{\mathop{\kern0pt s\lower.5pt\hbox{${}
     \scriptstyle'$}}\limits^{\kern0pt\raise.02ex
     \vbox to 0pt{\hbox{$\scriptscriptstyle\leftarrow$}\vss}}}}
     
\def\vtdash{{\mathop{\kern0pt t\lower0pt\hbox{${}
     \scriptstyle'$}}\limits^{\kern0pt\raise.1ex
     \vbox to 0pt{\hbox{$\scriptscriptstyle\rightarrow$}\vss}}}}
     
\def\tvdash{{\mathop{\kern0pt t\lower0pt\hbox{${}
     \scriptstyle'$}}\limits^{\kern0pt\raise.1ex
     \vbox to 0pt{\hbox{$\scriptscriptstyle\leftarrow$}\vss}}}}
     
\def\vddash{{\mathop{\kern0pt d\raise1pt\hbox{${}
     \scriptstyle'$}}\limits^{\kern0pt\raise.02ex
     \vbox to 0pt{\hbox{$\scriptscriptstyle\rightarrow$}\vss}}}}
     
\def\dvdash{{\mathop{\kern0pt d\raise1pt\hbox{${}
     \scriptstyle'$}}\limits^{\kern0pt\raise.02ex
     \vbox to 0pt{\hbox{$\scriptscriptstyle\leftarrow$}\vss}}}}
     
\def\vtstar{{\mathop{\kern0pt t\raise2.5pt\hbox{${}
     \scriptstyle*$}}\limits^{\kern0pt\raise.1ex
     \vbox to 0pt{\hbox{$\scriptscriptstyle\rightarrow$}\vss}}}}
     
\def\tvstar{{\mathop{\kern0pt t\raise2.5pt\hbox{${}
     \scriptstyle*$}}\limits^{\kern0pt\raise.1ex
     \vbox to 0pt{\hbox{$\scriptscriptstyle\leftarrow$}\vss}}}}

\def\vtstarD{{\mathop{\kern0pt t\kern.5pt\raise3pt\hbox{${}
     \scriptstyle*$}{\kern-5.5pt\lower3pt\hbox{$
     \scriptstyle D$}}}\limits^{\kern0pt\raise.1ex
     \vbox to 0pt{\hbox{$\scriptscriptstyle\rightarrow$}\vss}}}}

\def\tvstarD{{\mathop{\kern0pt t\kern.5pt\raise3pt\hbox{${}
     \scriptstyle*$}{\kern-5.5pt\lower3pt\hbox{$
     \scriptstyle D$}}}\limits^{\kern0pt\raise.1ex
     \vbox to 0pt{\hbox{$\scriptscriptstyle\leftarrow$}\vss}}}}

\def\vt{\lowfwd t{1.5}1}
\def\tv{\lowbkwd t{1.5}1}
\def\vT{{\fwd T1}}
\def\vO{\mathcal{O}}

\def\vOc{\overline{\mathcal{O}}}

\def\vt{\lowfwd t{1.5}1}
\def\tv{\lowbkwd t{1.5}1}

\def\vx{\lowfwd x{1.5}1}
\def\xv{\lowbkwd x{1.5}1}

\def\va{\lowfwd a{1.5}1}

\def\vb{\lowfwd b{1.5}1}
\def\bv{\lowbkwd b{1.5}1}
\def\vE{\lowfwd E{1.5}1}

\begin{document}

\title{Representations of infinite tree-sets}

\author{J. Pascal Gollin and Jakob Kneip}
\email{\tt pascal.gollin@uni-hamburg.de}
\email{\tt jakob.kneip@studium.uni-hamburg.de}
\address{Fachbereich Mathematik, Universit\"{a}t Hamburg, Bundesstra{\ss}e 55, 20146 Hamburg, Germany}

\begin{abstract} 
    Tree sets are abstract structures that can be used to model various tree-shaped objects in combinatorics. Finite tree sets can be represented by finite graph-theoretical trees.
    We extend this representation theory to infinite tree sets. 
    
    First we characterise those tree sets that can be represented by tree sets arising from infinite trees; these are precisely those tree sets without a chain of order type ${\omega+1}$. 
    Then we introduce and study a topological generalisation of infinite trees which can have limit edges, and show that every infinite tree set can be represented by the tree set admitted by a suitable such tree-like space.
\end{abstract}

\maketitle

\section{Introduction}
\label{sec:introduction}

Separations of graphs have been studied in the context of structural graph theory for a long time. 
For instance every edge of the decomposition tree of a tree-decomposition of a graph defines a separation in a natural way\footnote{As the sides of the separation, consider the union of the parts corresponding to the components of the tree after deleting the edge.}. 
The separations obtained in this way have an additional important property: they are nested\footnote{Two separations are nested if a side of the first separation is a subset of a side of the second separation, and the other side of the second separation is a subset of the other side of the first separation.} 
with each other. 
Looking at nested sets of separations of a graph has since been a useful way to study tree-decompositions, and especially in infinite graphs they offer an analogue when a tree-decomposition with a certain desired property may not exist (see \cite{RST-infinite-minors} for example).

While any tree-decomposition of a graph into small parts witnesses that the graph has low tree-width, there are various dense objects that force high tree-width in a graph. 
Among these are large cliques and clique minors, large grids and grid minors as well as high-order brambles. 
All these dense objects in a graph have the property that they orient its low-order separations by lying mostly on one side of any given low-order separation. 
For such a dense structure in a graph these orientations of separations are consistent with each other: 
no two of them `disagree' about where the dense object lies by pointing away from each other.

In~\cite{RSTangles} Robertson and Seymour proposed the notion of tangles, which are such families of consistently oriented separations up to a certain order. 
These tangles can be studied in their own right, instead of any dense objects that may induce them. 
By varying the strength of the consistency conditions one can model different kinds of dense objects, and the resulting consistent orientations give rise to different types of tangles.

To talk about these separations systems one does not even need an underlying graph structure or ground set: 
they can be formulated in a purely axiomatic way, see Diestel~\cite{ASS}. 
Such a separation system is simply a partially ordered set with an order-reversing involution.
The notions of consistency of separations that come from dense substructures in graphs can be translated into this setting as well. 
The tangles of graphs then become abstract tangles, and the tree-like structures become nested systems of separations, so-called \emph{tree sets} \cite{TreeSets}. 
This abstract framework turns out to be no less powerful, even for graphs alone, than ordinary graph separations. 
In~\cite{Duality} Diestel and Oum established an abstract duality theorem for separation systems which easily implies (see~\cite{DualityII}) all the classical duality results from graph- and matroid theory, such as the tree-width duality theorem by Seymour and Thomas~\cite{TreeWidth}. 
The unified duality theorem asserts that for any sensible notion of consistency a separation system contains either an abstract tangle or a tree set witnessing that no such tangle exists. 
Furthermore this abstract notion of separation systems can be applied in fields outside of graph theory, for instance in image analysis~\cite{MonaLisa}.

Tree sets are also interesting objects in their own right: 
they are flexible enough to model a whole range of other `tree-like' structures in discrete mathematics, 
such as ordinary graph trees, order trees and nested systems of bipartitions of sets \cite{TreeSets}. 

In fact, tree sets and graph-theoretic trees are related even more closely than that: 
for any tree $T$ 
the set $\vE$ of oriented edges of $T$ admits a natural partial order, 
which in fact turns $\vE$ into a tree set, the \emph{edge tree set of~$T$}. 
As was shown in~\cite{TreeSets}, these edge tree sets of graph-theoretical trees are rich enough to represent all finite tree sets: 
every finite tree set is isomorphic to the edge tree set of a suitable tree.

In this paper we extend the analysis of representations of tree sets to infinite tree sets. 
The definition of an edge tree set of a graph-theoretical tree straightforwardly extends to infinite trees.
From the structure of these it is clear that the edge tree set of a tree $T$ cannot contain a chain of order type~${\omega+1}$. 
We will show that this is the only obstruction for a tree set to being representable by the edge tree set of a (possibly infinite) tree:

\begin{thm-intro}
    \label{thm-intro-1}
    Every tree set without a chain of order type~${\omega+1}$ is isomorphic to the edge tree set of a suitable tree.
\end{thm-intro}

Secondly, we would like to represent infinite tree sets that do contain a chain of order type~${\omega+1}$ by edge tree sets of an adequate tree structure as well. 
To achieve this we turn to the notion of \emph{graph-like spaces} introduced by Thomassen and Vella~\cite{OriginalGLS} and further studied by Bowler, Carmesin and Christian~\cite{GLS}: 
these are topological spaces with a clearly defined structure of vertices and edges, which can be seen as a limit object of finite graphs. 
In particular, for a chain of any order type, there exists a graph-like space containing a `path' whose edges form a chain of that order type. 
Therefore the \emph{tree-like spaces}, those graph-like spaces which have a tree-like structure, overcome the obstacle of chains of order type ${\omega+1}$ which prevented the edge tree sets of infinite trees from representing all infinite tree sets: unlike graph-theoretic trees, tree-like spaces can have limit edges.
And indeed we will prove in this paper that the edge tree sets of tree-like spaces can be used to represent all tree sets.

\begin{thm-intro}
    \label{thm-intro-2}
    Every tree set is isomorphic to the edge tree set of a suitable tree-like space.
\end{thm-intro}

This paper is organised as follows. 
In Section~\ref{sec:separationsystems} we recall the basic definitions of abstract separation systems and tree sets and establish a couple of elementary lemmas we will use throughout the paper. 
Following that, in Section~\ref{sec:tame-tree-sets}, we formally define the edge tree set of a tree and prove Theorem~\ref{thm-intro-1}. 
In Section~\ref{sec:treelikespaces}, we introduce the concept of {\em tree-like spaces} which generalise infinite graph-theoretical trees. 
We define edge tree sets of tree-like spaces analogously to edge tree sets of graph-theoretical trees and then prove Theorem~\ref{thm-intro-2}.

\section{Separation systems}
\label{sec:separationsystems}

An abstract \emph{separation system} $\vS=(\vS,\le,{}^*)$ is a partially ordered set with an order-reversing involution ${}^*$. 
An element $\vs\in\vS$ is called an \emph{oriented separation}, 
and its \emph{inverse}~${(\vs)^*}$ is denoted as~$\sv$, and vice versa. 
The pair ${s=\{\vs,\sv\}}$ is an \emph{unoriented separation}\footnote{To improve readability `oriented' and `unoriented' will often be omitted if the type of separation follows from the context.}, 
with \emph{orientations} $\vs$ and $\sv$, and the set of all such pairs is denoted as~$S$. 
The assumption that~${}^*$ is order-reversing means that for all~${\vs, \vr \in \vS}$ we have 
${\vs \leq \vr}$ if and only if~${\sv \geq \rv}$.
If $S'$ is a set of unoriented separations, we write~$\vSdash$ for the set~${\bigcup S'}$ of all orientations of separations in~$S'$.

A separation~$\vs$ is \emph{small} and its inverse~$\sv$ \emph{co-small} if ${\vs\le\sv}$. 
If neither~$\vs$ nor~$\sv$ is small then~$s$ is \emph{regular}, and we call both~$\vs$ and~$\sv$ regular as well.

A separation~${\vs \in \vS}$ is \emph{trivial in}~$\vS$ and its inverse~$\sv$ is \emph{co-trivial in}~$\vS$ if there is some~${\vr \in \vS}$ with~${\vs \leq \vr,\rv}$ and~${s \neq r}$. 
In this case~$r$ is the \emph{witness} of the triviality of~$\vs$. 
If neither~$\vs$ nor~$\sv$ is trivial in~$\vS$ we call~$s$ \emph{nontrivial}. 
If~$\vs$ is a trivial separation with witness~$r$ then~$\vs$ is small as~${\vs \leq \vr \leq \sv}$. Conversely every separation that lies below a small separation is trivial: 
if~$\vs$ is small and~${r\ne s}$ has an orientation~${\vr \leq \vs}$, then~$\vr$ is trivial as ${\vr < \vs \leq \sv}$.

Two unoriented separations~$s$ and~$r$ are \emph{nested} if they have comparable orientations. 
Otherwise~$r$ and~$s$ \emph{cross}. 
A set~$S'$ of separations is nested if all of its elements are pairwise nested.

A \emph{tree set} is a nested separation system with no trivial elements. 
It is \emph{regular} if all of its elements are regular, i.e.\ if no~${\vs \in \tau}$ is small.

An \emph{orientation} of a set~$\vSdash$ or~$S'$ of separations is a set~${O \subseteq \vSdash}$ with~${|O \cap s| = 1}$ for every~${s \in S'}$. 
An orientation is \emph{consistent}~if ${\sv \leq \vr}$ implies ${r=s}$ for all~${\vr,\vs\in O}$. 
A \emph{partial orientation} of~$\vS$ is an orientation of a subset of~$\vS$. 
A partial orientation~$P$ \emph{extends} another partial orientation~$Q$ if~${Q \subseteq P}$.

For a tree set~$\tau$ an orientation~$O$ of~$\tau$ is \emph{splitting} if it is consistent and has the property that for every~${\vr\in O}$ there is some maximal element~$\vs$ of~$O$ with~${\vr \leq \vs}$. 

Consistent orientations of a tree set~$\tau$ can be thought of as the `vertices' of a tree set, an idea that we will make more precise in the next sections. 
In the context of infinite tree sets, the non-splitting orientations can be thought of as `limit vertices' or `ends' of the tree set.

A subset~${\sigma \subseteq \tau}$ is a \emph{star} if~${\vr \leq \sv}$ for all~${\vr,\vs\in\sigma}$ with~${\vr \neq \vs}$. 
For example, the set of maximal elements of a consistent orientation of a tree set is always a star:

\begin{lemma}
    \label{lem:cons_star}
    Let~$O$ be a consistent orientation of a tree set~$\tau$. 
    Then the set~$\sigma$ of the maximal elements of~$O$ is a star.
\end{lemma}

\begin{proof}
    Let~${\vr, \vs \in \sigma}$ with~${\vr \neq \vs}$ be given. 
    Then neither~${\vr \leq \vs}$ nor~${\vr\ge\vs}$ as both are maximal elements of~$O$. 
    The consistency of~$O$ implies that~${\vr \not\ge \sv}$, so~${\vr \leq \sv}$ is the only possible relation and hence~$\sigma$ is a star.
\end{proof}

A star~${\sigma \subseteq \tau}$ \emph{splits}~$\tau$, 
or is a \emph{splitting star} of~$ \tau $, 
if it is the set of maximal elements of a splitting orientation of~$\tau$. 
Note that every element of a finite tree set lies in a splitting star, but infinite tree sets can have elements that lie in no splitting star; 
see Example~\ref{ex:split-omega} and Lemma~\ref{lem:split-omega} below.

More generally, given a partial orientation~$P$ of~$\tau$, is it possible to extend it to a consistent orientation of~$\tau$? 
Of course~$P$ needs to be consistent itself for this to be possible. 
The next Lemma shows that under this necessary assumption it is always possible to extend a partial orientation to all of~$\tau$. 
In particular, every element of a tree set induces a consistent orientation in which it is a maximal element. 
This orientation is in fact unique:

\begin{lemma}[Extension Lemma]
    \label{lem:extension}
    \cite{ASS}
    Let~$S$ be a set of separations, and let~$P$ be a consistent partial orientation of~$S$.
    \begin{enumerate}
        [label=\textnormal{(\roman*)}]
        \item\label{item:ext-1} $P$ extends to a consistent orientation~$O$ of~$S$ if and only if no element of $P$ is co-trivial in~$S$.
        \item\label{item:ext-2} If~$\vp$ is maximal in~$P$, then~$O$ in~\ref{item:ext-1} can be chosen with~$\vp$ maximal in~$O$ if and only if~$\vp$ is nontrivial in~$\vS$.
        \item\label{item:ext-3} If~$S$ is nested, then the orientation~$O$ in \ref{item:ext-2} is unique.
    \end{enumerate}
\end{lemma}

The last part of the Extension Lemma implies that every element~$\vs$ of a tree set~$\tau$ is maximal in exactly one consistent orientation~$O$ of~$\tau$. 
Hence~$\vs$ lies in a splitting star if and only if this~$O$ is splitting.

In an infinite tree set there might be elements that do not lie in a splitting star:

\begin{example}
    \label{ex:split-omega}
    Let~$\tau$ be the tree set with ground set
    \[ 
        \{ \vs_n \,|\, n \in \mathbb{N} \} \cup \{ \sv_n \,|\, n \in \mathbb{N} \} \cup \{ \vt, \tv \}, 
    \]
    where~${\vs_i \leq \vs_j}$ and~${\sv_i \geq \sv_j}$ whenever~${i\leq j}$, 
    as well as~${\vs_n \leq \vt}$ and~${\sv_n \geq \tv}$ for all~${n \in \mathbb{N}}$. 
    The separation~$\tv$ is maximal in the orientation
    \[ 
        O = \{ \vs_n \,|\, n \in \mathbb{N} \} \cup \{ \tv \}, 
    \]
    which is not splitting as no~$\vs_n$ lies below a maximal element of~$O$. 
    Hence~$\tv$ does not lie in a splitting star of~$\tau$.
\end{example}

In the above example the chain~${C = \{ \vs_n \,|\, n \in \mathbb{N} \} \cup \{\vt\}}$ has order-type~$ {\omega+1}$. 
But these $\omega+1$ chains turn out to be the only obstruction for separations not being elements of splitting stars, as the following lemma shows. 
Let us call a tree set that does not contain a chain of order type~${\omega+1}$ \emph{tame}.

\begin{lemma}
    \label{lem:split-omega}
    Every element of a tame tree set~$\tau$ lies in some splitting star of~$\tau$.
\end{lemma}

\begin{proof}
    For every~${\vt \in \tau}$ we can apply the Extension Lemma~\ref{lem:extension} to~${P := \{ \vt \}}$ to find that there is a unique consistent orientation~$O$ 
    of~$\tau$ in which~$\vt$ is a maximal element. 
    Thus~$\vt$ lies in a splitting star if and only if this orientation~$O$ is splitting. 
    Let us show that for every~${\vt \in \tau}$ this orientation~$O$ splits~$\tau$ unless~$O$ contains a chain of order type~$\omega$ for which~$\tv$ is an upper bound; 
    this directly implies the claim since every such chain in~$O$ together with~$\tv$ is a chain of order type~${\omega+1}$ in~$\tau$.

    So let~${\vt \in \tau}$ be given and consider the unique consistent orientation~$O$ of~$\tau$ in which~$\vt$ is maximal. 
    Suppose that~$O$ does not split~$\tau$, i.e.\ that there is some~${\vs \in O}$ which does not lie below any maximal element of~$O$. 
    Consider the set~${C \subseteq O}$ of all elements~$\vr$ of~$O$ with~${\vr \geq \vs}$. 
    Since~$\vs$ and hence no element of~$C$ can lie below~$\vt$ we must have~${\vr \leq \tv}$ for all~${\vr \in C}$ since $\tau$ is nested. 
    Thus~$\tv$ is an upper bound for~$C$. 
    Now if~$C$ has a maximal element then this separation is also a maximal element of~$O$, contrary to our assumption about~$\vs$; 
    therefore~$C$ cannot have a maximal element and hence contains a chain of order type~$\omega$, as claimed.
\end{proof}

A direct consequence of Lemma~\ref{lem:split-omega} is that every element of a finite tree set lies in a splitting star.

Given two separation systems~$R$ and~$S$, 
a map~${f \colon R \to S}$ is a \emph{homomorphism} of separation systems if it \textit{commutes with the involution}, 
i.e.\ ${(f(\vr))^*=f(\rv)}$ for all~${\vr \in R}$, 
and is \textit{order-preserving}, 
i.e.\ ${f(\vr_1) \leq f(\vr_2)}$ whenever~${\vr_1 \leq \vr_2}$ for all~${ \vr_1, \vr_2 \in R}$.
Please note that the condition for~$f$ to be order-preserving is not `if and only if': 
it is allowed that~${f(\vr_1) \leq f(\vr_2)}$ for incomparable~${\vr_1, \vr_2 \in R}$. 
Furthermore~$f$ need not be injective.

As all trivial separations are small every regular nested separation system is a tree set. 
These two properties, regular and nested, are preserved by homomorphisms of separations systems, albeit in different directions: 
the image of nested separations is nested, and the preimage of regular separations is regular.

\begin{lemma}
    \label{lem:beforeIso}
    Let~${f \colon R \to S}$ be a homomorphism of separation systems. 
    If~$S$ is regular then so is~$R$; and if~$R$ is nested then so is its image in~$S$.
\end{lemma}

\begin{proof}
    First suppose that some~${\vr \in R}$ is small, that is, that~${\vr \leq \rv}$. 
    Then
    \[ 
        f(\vr) \leq f(\rv) = (f(\vr))^*, 
    \]
    so~$S$ contains a small element. 
    Therefore if~$S$ is regular then $R$ must be regular as well.

    Now suppose that~$R$ is nested consider two unoriented separations~${s, s' \in S}$ and for which there are~${r, r' \in R}$ with~${s = f(r)}$ and~${s' = f(r')}$. 
    Since~$R$ is nested~$r$ and~$r'$ have comparable orientations, say~${\vr \leq \vrdash}$. 
    Then~${\vs := f(\vr) \leq f(\vrdash) =: \vsdash}$, showing that~$s$ and~$s'$ are nested. 
    Hence if~$R$ is nested its image in~$S$ is nested too.
\end{proof}

A bijection~${f \colon R \to S}$ is an \emph{isomorphism} of separation systems if both~$f$ and its inverse map are homomorphisms of separation systems. 
Two separation systems~$R$ and~$S$ are \emph{isomorphic}, denoted as~${R \cong S}$, if there is an isomorphism ${f\colon R \to S}$ of separation systems. 
If one of~$R$ and~$S$ (and thus both) is a tree set we call~$f$ an \emph{isomorphism of tree sets}.

Lemma~\ref{lem:beforeIso} makes it possible to show that a homomorphism~${f \colon R \to S}$ of separation systems is an isomorphism of tree sets without knowing beforehand that either~$R$ or~$S$ is a tree set:

\begin{lemma}
    \label{lem:Isomorphism}
    Let ${f \colon R \to S}$ be a bijective homomorphism of separation systems. 
    If~$R$ is nested and~$S$ regular then~$f$ is an isomorphism of tree sets.
\end{lemma}

\begin{proof}
    From Lemma~\ref{lem:beforeIso} it follows that both~$R$ and~$S$ are regular and nested, which means they are regular tree sets. 
    Therefore all we need to show is that the inverse of~$f$ is order-preserving, i.e.\ that~${\vr_1 \leq \vr_2}$ whenever~${f(\vr_1) \leq f(\vr_2)}$.
    Let~${\vr_1, \vr_2 \in R}$ with~${f(\vr_1) \leq f(\vr_2)}$ be given. 
    As~$R$ is nested,~$r_1$ and~$r_2$ have comparable orientations. 
    If~${\vr_1 \geq \vr_2}$, then~${f(\vr_1) = f(\vr_2)}$, implying~${\vr_1 = \vr_2}$ and hence the claim. 
    If~${\vr_1 \leq \rv_2}$, then~${f(\vr_1) \leq f(\vr_2),f(\rv_2)}$, contradicting the fact that~$S$ is a regular tree set. 
    Finally, if~${\vr_1 \geq \rv_2}$, then~${f(\rv_2) \leq f(\vr_2)}$, contradicting the fact that~$S$ is regular. 
    Hence~${\vr_1 \leq \vr_2}$, as desired.
\end{proof}

\section{Regular tame tree sets and graph-theoretical trees}
\label{sec:tame-tree-sets}

Every graph-theoretical tree~$T$ naturally gives rise to a tree set, its \emph{edge tree set}~$\tau(T)$ of~$T$ (see below for a formal definition). 
However, while every tree gives rise to a tree set, not every tree set `comes from' a tree. 
In this section we characterise those infinite tree sets that arise from graph-theoretical trees as the tree sets which are both regular and tame, i.e.~contain no chain of order-type~${\omega+1}$. 
More precisely, given a regular tame tree set~$\tau$ we will define a corresponding tree~${T(\tau)}$. 
These definitions in turn should be able to capture the essence of what it means to be `tree-like'. 
More precisely we want the following properties:
\begin{itemize}
    \item the tree constructed from the edge tree set of~$T$ is isomorphic to~$T$;
    \item the edge tree set of the tree constructed from~$\tau$ is isomorphic to~$\tau$.
\end{itemize}

\subsection{The edge tree set of a tree}
\label{sec:edge-tree-set}

Let~${T = (V,E)}$ be a graph-theoretical tree, finite or infinite. 
Let~$\vE(T)$ be the set of oriented edges of~$T$, that is
\[
    \vE(T) = \big\{ (x,y) \,\big|\, \{ x,y \} \in E(T) \big\}.
\]
We define an involution ${}^*$ by setting ${(x,y)^* := (y,x)}$ for all edges~${xy \in E(T)}$, and 
a partial order~$\leq$ on~$\vE(T)$ by setting ${(x,y) < (v,w)}$ for edges~${xy, vw \in E(T)}$ 
if and only if~${\{ x,y \} \neq \{ v,w \}}$ and the unique $\{x,y\}$--$\{v,w\}$-path in~$T$ joins~$y$ to~$v$. 
Then the \emph{edge tree set}~${\tau(T)}$ is the separation system~${(\vE(T),\le,*)}$. 
It is straightforward to check that~$\tau(T)$ is indeed a regular tree set.

Note that every maximal chain in~${\tau(T)}$ corresponds to the edge set of a path, ray or double ray in~$T$. 
Hence $\tau(T)$ does not contain any chain of length~${\omega+1}$ and hence is tame.

\vspace{0.2cm}

If~$T$ is the decomposition tree of a tree-decomposition of a graph~$G$, then the tree set~$\tau(T)$ is isomorphic to the tree set formed by the separations of~$G$ that correspond\footnote{An edge~$e$ of the decomposition tree~$T$ of a tree-decomposition naturally defines a graph separation by considering the union of the parts in the respective components of $T - e$ as the sides of that separation.} to the edges of~$T$ (with some pathological exceptions).
This relationship between tree-decompositions and tree sets was further explored in~\cite{TreeSets}.

\subsection{The tree of a regular tame tree set}
\label{sec:tree-of-tame-tree-set}

Let~$\tau$ be a regular tame tree set. 
Our aim is to construct a corresponding graph-theoretical tree $T(\tau)$.
Recall that a consistent orientation~$O$ of~$\tau$ is called splitting if every element of~$O$ lies below some maximal element of~$O$. 
By the uniqueness part of the Extension Lemma~\ref{lem:extension}, every splitting star extends to exactly one splitting orientation. 
Write $\vOc$ for the set of all splitting orientations of~$\tau$. 
We will use $\vOc$ as the vertex set of $T(\tau)$. 
Moreover note that it will turn out that the non-splitting orientations will precisely correspond to the ends of $T(\tau)$.

Let us show first that, for any two splitting stars, each of them contains exactly one element that is inconsistent with the other star. 
We will later use this little fact when we define the edges of our tree.

\begin{lemma}
    \label{lem:flipping_stars}
    Let~${\sigma_1, \sigma_2}$ be two distinct splitting stars of~$\tau$ and~${O_2 \in \vOc}$ the orientation inducing~$\sigma_2$. 
    Then there is exactly one~${\vs \in \sigma_1}$ with~${\sv \in O_2}$.
\end{lemma}

\begin{proof}
    There is at least one such~$\vs$ as~$O_2$ does not induce $\sigma_1$. 
    For any two~${\vr, \vs \in \sigma}$ the set~${\{ \rv, \sv \}}$ is inconsistent, so there is at most one ${\vs \in \sigma_1}$ with~${\sv \in O_2}$.
\end{proof}

Note that this lemma holds for every tree set as the proof did not use any assumptions on~$\tau$.

Our assumption that~$\tau$ is tame implies the following sufficient condition for a consistent orientation to be splitting:

\begin{lemma}
    \label{lem:one_closed}
    Let~$O$ be a consistent orientation of~$\tau$ with at least one maximal element. 
    Then~$O$ splits~$\tau$.
\end{lemma}

\begin{proof}
    Let~$\vt$ be a maximal element of~$O$. 
    By Lemma~\ref{lem:split-omega}~$\vt$ lies in a splitting star of~$\tau$, 
    i.e.\ is a maximal element of a consistent orientation that splits~$\tau$. 
    By the Extension Lemma~\ref{lem:extension},~$O$ is the only consistent orientation of~$\tau$ of which~$\vt$ is a maximal element; 
    hence~$O$ must be splitting.
\end{proof}

Together with the Extension Lemma~\ref{lem:extension} this immediately implies the following:

\begin{corollary}
    \label{cor:unique_star}
    Every~${\vs \in \tau}$ lies in exactly one splitting star of~$\tau$. 
    Equivalently every~${\vs \in \tau}$ is maximal in exactly one consistent orientation~$O$ and~${O \in \vOc}$.
\end{corollary}

\begin{proof}
    For~${\vs \in \tau}$ apply the Extension Lemma~\ref{lem:extension} to~${\{ \vs \}}$ to obtain a unique consistent orientation~$O$ of~$\tau$ in which~$\vs$ is a maximal element. 
    It then follows from Lemma~\ref{lem:one_closed} that~$O$ is splitting.
\end{proof}

For~${\vs \in \tau}$ write~${O(\vs)}$ for the unique consistent orientation of~$\tau$ in which~$\vs$ is maximal. 
Then Lemma~\ref{lem:flipping_stars} together with Corollary~\ref{cor:unique_star} says that for distinct~${O, O' \in \vOc}$ there is at most one~${\vs \in O'}$ with~${O(\sv) = O}$.

\vspace{0.2cm}

Now we define the graph $T(\tau)$. 
Let~${V(T(\tau)) = \vOc}$ and
\[
    E(T(\tau)) = \big\{ \{ O(\vs), O(\sv) \} \,\big|\, \vs \in \tau \big\}.
\]
We call~$T(\tau)$ the \emph{tree corresponding to~$\tau$}, 
where~$\tau$ is a regular tame tree set.

First note that $T(\tau)$ does not contain any loops and hence is indeed a simple graph since~$O(\vs)$ and $O(\sv)$ are different for any~${\vs \in \tau}$.

We need to check that~$T(\tau)$ is a tree. 

\begin{lemma}
    \label{lem:T_acyclic}
    $T(\tau)$ does not contain any cycles.
\end{lemma}

\begin{proof}
    For~${O \in \vOc}$ the set of incoming edges is precisely the splitting star induced by~$O$. 
    If~${\vs_1, \dots, \vs_k}$ are the edges of an oriented cycle in~$\vT$, 
    then each of these and the inverse of its cyclic successor lie in a common splitting star. 
    Hence~${\vs_1 \leq \vs_2 \leq \dots \leq \vs_k \leq \vs_1}$ by the star property, a contradiction.
\end{proof}

To prove that~$T(\tau)$ is connected, our strategy is as follows. 
To find a path from~${O \in \vOc}$ to~${O' \in \vOc}$ we use Lemma~\ref{lem:flipping_stars} to find~${\vs \in O}$ which is maximal in~$O$ with~${\sv \in O'}$. 
Then we consider~${O^* := (O \cup \{\sv\}) \setminus \{\vs\}}$.
This orientation is again in~$\vOc$ and a neighbour of~$O$ in~$T(\tau)$. 
If~${O^* = O'}$ we are done; 
otherwise we can iterate the process with~$O^*$ and~$O'$. 
Either this process terminates after finitely many steps, in which case we found a path from~$O$ to~$O'$, or it continues indefinitely. 
In the latter case the infinitely many separations we inverted form a chain with an upper bound in~$O'$, which would yield a chain of order type~${\omega+1}$.

The next short Lemma forms the basis of this iterative flipping process.

\begin{lemma}
    \label{lem:infinite_flips}
    Let~${\vs_1, \dots, \vs_n, \vsdash \in \tau}$ be distinct separations with~${O(\sv_{k+1}) = O(\vs_k)}$ for all ${k \in \mathbb{N}}$ with ${1 \leq k < n}$ and~${\vs_n < \vsdash}$. 
    Then there is a~${\vs_{n+1} \in \tau}$ with~${O(\sv_{n+1}) = O(\vs_n)}$ and~${\vs_{n+1} \leq \vsdash}$.
\end{lemma}

\begin{proof}
    Let~$\vs_{n+1}$ be the unique separation in~${O(\vsdash)}$ with~${O(\sv_{n+1}) = O(\vs_n)}$. 
    Then~${\vs_n \leq \vs_{n+1}}$ by the star property. 
    Hence if~${\vs_{n+1} \leq \svdash}$, then~$\vs_n$ would be trivial, therefore~${\vs_{n+1} \leq \vsdash}$ as desired.
\end{proof}

For~${\vs_1, \dots, \vs_n, \vsdash}$ and~$\vs_{n+1}$ as in Lemma~\ref{lem:infinite_flips} there is an edge between~${O(\vs_k)}$ and~${O(\sv_{k+1})}$ for every~${1 \leq k \leq n}$. 
Additionally if~${\vs_{n+1} \neq \vsdash}$ then~${\vs_1, \dots, \vs_{n+1}, \vsdash}$ again fulfill the assumptions of the lemma, so it can be used iteratively.

Furthermore note that~${\vs_1 \leq \vs_2 \leq \dots \leq \vs_n \leq \vs_{n+1}}$, so if this iteration does not terminate the~$\vs_k$ form an infinite chain. 
From this we now prove that~$T(\tau)$ is connected.

\begin{lemma}
    \label{lem:T_connected}
    $T(\tau)$ is connected.
\end{lemma}

\begin{proof}
    Let~${O, O'\in \vOc}$ be distinct orientations. 
    Let~$\vs_1$ be the unique separation in~$O'$ with~${O = O(\sv_1)}$, and~$\svdash$ the unique separation in~$O$ with~${O' = O(\vsdash)}$. 
    Then~${\vs_1 \leq \vsdash}$, and if~${\vs_1 = \vsdash}$ then~$O$ and~$O'$ are joined by an edge in~$T(\tau)$. 
    Otherwise the assumptions of Lemma~\ref{lem:infinite_flips} are met for~${n = 1}$. 
    Applying Lemma~\ref{lem:infinite_flips} iteratively either yields ${\vs_{n+1} = \vsdash}$ for some~${n \in \mathbb{N}}$, in which case we found a path in~$T(\tau)$ joining~$O$ and~$O'$, 
    or we obtain a strictly increasing sequence~${(\vs_n)_{n\in\mathbb{N}}}$ with~${\vs_n \leq \vsdash}$ for all~${n \in \mathbb{N}}$, that is, a chain of order type~${\omega+1}$.
\end{proof}

\subsection{Regular tame tree sets and trees -- A characterisation}\label{sec:treeschar}

Finally we will prove that the given constructions of the previous subsections agree with each other.

\begin{lemma}
    \label{lem:tau-T-tau-iso}
    Any regular tame tree set~$\tau'$ us isomorphic to~${\tau(T(\tau'))}$.
\end{lemma}

\begin{proof}
    Let~${\phi \colon \tau' \to \tau(T(\tau'))}$ be the map defined by~${\phi(\vs) = ( O(\sv), O(\vs) )}$. 
    This is a bijection by Corollary~\ref{cor:unique_star}. 
    Note that for~${\vs \in \tau'}$ the orientations~${O(\sv)}$ and~${O(\vs)}$ differ only in~$s$ by consistency and are thus adjacent in~$T$.
    
    As~$\tau'$ and~${\tau(T(\tau'))}$ are regular tree sets all we need to show is that~$\phi$ is a homomorphism of separation systems. 
    Then~$\phi$ will be an isomorphism of tree sets by Lemma~\ref{lem:Isomorphism}.
    
    It is clear from the definition that~$\phi$ commutes with the involution. 
    Therefore it suffices to show that~$\phi$ is order-preserving.
    
    Let~${\vs, \vsdash \in \tau'}$ be two separations with~${\vs < \vsdash}$. 
    We need to show that the unique ${\{ O(\sv), O(\vs) \}}$--${\{ O(\svdash), O(\vsdash) \}}$-path in~${T(\tau)}$ joins~${O(\vs)}$ and~${O(\svdash)}$. 
    Redoing the proof of Lemma~\ref{lem:T_connected} with~${O = O(\vs)}$ and~${O' = O(\svdash)}$ constructs a $O(\vs)$--$O(\svdash)$-path every one of whose nodes contains~$\vs$ and~$\svdash$ by consistency. 
    Hence~${\phi(\vs) < \phi(\vsdash)}$ as desired.
\end{proof}

\begin{lemma}
    \label{lem:T-tau-T-iso}
    Any graph-theoretic tree~$T'$ is isomorphic to~${T(\tau(T'))}$.
\end{lemma}

\begin{proof}
    If ${|V(T')| = 1}$, then~${\tau(T')}$ is empty and hence~${|V(T(\tau(T')))| = 1}$.
    
    Otherwise, for each node~${v \in V(T')}$ there is at some oriented edge~${(w,v) \in \vec E(T')}$ pointing towards that node. 
    Let~${\phi \colon T' \to T(\tau(T'))}$ be the map defined by~${\phi(v) := {O}((w,v))}$.
    This map is well-defined since the edges directed towards a node~${v \in V(T')}$ form a splitting star with the same maximal elements yielding the unique consistent orientation containing all these oriented edges (cf.~Corollary~\ref{cor:unique_star}).
    
    Similarly, given some~${O = O((w,v)) \in V(T(\tau(T')))}$, we obtain~${\phi(v) = O}$ and hence that~$\phi$ is surjective.
    By construction there is an edge between~${O((v,w))}$ and~${O((w,v))}$ for any edge~${vw \in E(T)}$ and similarly no edge between~${O((v,w))}$ and~$O$ if~${(w,v)}$ is not maximal in~$O$.
\end{proof}

Hence we have proven our main theorem of this section:

\begin{theorem}
    \label{thm:tametreesets}
    \begin{enumerate}
        \item A tree set is isomorphic to the edge tree set of a tree if and only if it is regular and tame.
        \item Any regular and tame tree set~$\tau'$ is isomorphic to~${\tau(T(\tau'))}$.
        \item Any graph-theoretic tree~$T'$ is isomorphic to~${T(\tau(T'))}$.\qed
    \end{enumerate}
\end{theorem}

Additionally, for distinct but comparable tree sets, we can say precisely in which way the corresponding trees from Theorem~\ref{thm:tametreesets} above are comparable: one will be a minor of the other.

\begin{theorem}
    \label{thm:treeminors}
    Let~$T_1$,~$T_2$ be trees and~$\tau_1$,~$\tau_2$ be regular tame tree-sets.
    \begin{enumerate}
        \item If~${\tau_1 \subseteq \tau_2}$, then~${T(\tau_1)}$ is a minor of~${T(\tau_2)}$.
        \item If~$T_1$ is a minor of~$T_2$, then~${\tau(T_1)}$ is isomorphic to a subset of~${\tau(T_2)}$.
    \end{enumerate}
\end{theorem}

Theorem~\ref{thm:treeminors} is a special case of Theorems~\ref{thm:tlsminor1} and~\ref{thm:tlsminor2} from the next section and hence we will omit its proof here.

\section{Regular tree sets and tree-like spaces}
\label{sec:treelikespaces}

\subsection{Graph-like spaces}
\label{sec:graphlikespaces}

As we have seen in Section~\ref{sec:tame-tree-sets}, not every tree set, even regular, can be represented as the edge tree set of a tree. 
In this section we find a (topological) relaxation of the notion of a (graph-theoretical) tree, 
to be called \emph{tree-like spaces}.
Like trees, these tree-like spaces give rise to a regular edge tree set in a natural way, but which are just general enough that, conversely, every regular tree set can be represented as the edge tree set of a tree-like space.

The concept of graph-like spaces was first introduced in~\cite{OriginalGLS} by Thomassen and Vella, and further studied in~\cite{GLS} by Bowler, Carmesin and Christian.
In~\cite{GLS} the authors discuss the connections between graph-like spaces and graphic matroids, which are of no interest to us here. 
Instead we determine when a graph-like space is tree-like, and then show that every regular tree set can be represented as the edge tree set of a tree-like space.

Graph-like spaces are limit objects of graphs that are not themselves graphs. 
In short they consist of the usual vertices and edges, together with a topology that allows the vertices and edges to be limits of each other. 
The formal definition is as follows.

\begin{definition}
    \label{def:GLS}
    \cite{GLS}
    A \emph{graph-like space}~$G$ is a topological space (also denoted by~$G$) 
    together with a \emph{vertex set}~${V(G)}$, 
    an \emph{edge set}~${E(G)}$ 
    and for each~${e \in E(G)}$ a continuous map~${\iota_e^G:[0,1]\to G}$ 
    (the superscript may be omitted if~$G$ is clear from the context) such that:
    \begin{itemize}
        \item The underlying set of~$G$ is~${V(G) \dot{\cup} [(0,1)\times E(G)]}$.
        \item For any~${x \in (0,1)}$ and~${e \in E(G)}$ we have~${\iota_e(x)=(x,e)}$.
        \item ${\iota_e(0)}$ and~${\iota_e(1)}$ are vertices (called the \emph{end-vertices} of~$e$).
        \item ${\iota_e\upharpoonright_{(0,1)}}$ is an open map.
        \item For any two distinct~${v, v' \in V(G)}$, there are disjoint open subsets~${U, U'}$ of~$G$ partitioning~$V(G)$ and with~${v \in U}$ and~${v' \in U'}$.
    \end{itemize}
    The \emph{inner points} of the edge~$e$ are the elements of~${(0,1) \times \{ e \}}$.
\end{definition}

Note that~$G$ is always Hausdorff. 
For an edge~${e \in E(G)}$ the definition of graph-like space allows~${\iota_e(0)=\iota_e(1)}$. 
We call such an edge a \emph{loop}. 
In our discussions of graph-like spaces loops are irrelevant, so the reader may imagine all graph-like spaces to be loop-free.

If~$U$ and~$U'$ are disjoint open subsets of~$G$ partitioning~$V(G)$ 
we call the set of edges with end-vertices in both~$U$ and~$U'$ a \emph{topological cut} of~$G$ and say that the pair~${(U,U')}$ \textit{induces} that cut. 
The last property of graph-like spaces then says that any two vertices can be separated by a topological cut.
Since the set of inner points of any edge is open, we immediately get the following remark.

\begin{remark}
    \label{rem:finite-top-cuts}
    Any topological cut of a compact graph-like space is finite. \qed
\end{remark}

A graph-like space~$G'$ is a \emph{sub-graph-like space} of a graph-like space~$G$ if 
${V(G') \subseteq V(G)}$, 
${E(G') \subseteq E(G)}$ 
and~$G'$ is a subspace of~$G$ (as topological spaces).
By slight abuse of notation we will write ${G' \subseteq G}$ to say that~$G'$ is a sub-graph-like space of~$G$.

Let~$G$ be a graph-like space and ${F \subseteq E(G)}$ a set of edges of~$G$.
We write ${G-F}$ for the sub-graph-like space ${G \setminus \{ (x,e) \,|\, x \in (0,1), e \in F \}}$ with the same vertex set as $G$, with edge set ${E(G) \setminus F}$ and ${\iota^{G-F}_e = \iota^G_e}$ for all ${e \in E(G) \setminus F}$. 
We abbreviate ${G-\{e\}}$ as~${G-e}$.
Given a set ${W \subseteq V(G)}$ of non-end-vertices we write ${G-W}$ for the sub-graph-like space ${G \setminus W}$ with ${V(G-W) := V(G) \setminus W}$, ${E(G-W) := E(G)}$ and ${\iota^{G-F}_e = \iota^G_e}$ for all~${e \in E(G)}$. 

For reasons of cardinality arc-connectedness is not a very useful notion in graph-like spaces. 
Instead we work with an adapted concept of arcs. 

A graph-like space~$P$ is a \emph{pseudo-arc} 
if~$P$ is a compact connected graph-like space 
with a \emph{start-vertex}~$a$ 
and an \emph{end-vertex}~$b$ 
if for each ${e \in E(P)}$ the vertices~$a$ and~$b$ are separated in~${P-e}$.
If $P$ contains an edge then $a \ne b$; otherwise we call $P$ \emph{trivial}. 
A graph-like space~$G$ is \emph{pseudo-arc-connected} if for all vertices ${a,b \in V(G)}$ there is a pseudo-arc ${P \subseteq G}$ with start-vertex~$a$ and end-vertex~$b$.

Note that in the original definition of pseudo-arc given in \cite{GLS} an extra condition was given, which turns out to be redundant, as seen in the next lemma. 

\begin{lemma}
    \label{lem:pseudo-arc-redundancy}
    Let~$P$ be a pseudo-arc. 
    Then for every ${x, y \in V(P)}$ there is an edge~${e \in E(P)}$ such that~$x$ and $y$ are separated in~${P-e}$. 
\end{lemma}

\begin{proof}
    Let~${x, y \in V(P)}$.
    Let~$F$ be a minimal topological cut separating~$x$ and~$y$, which exists by Remark~
    \ref{rem:finite-top-cuts}. 
    If~${|F| = 1}$ we are done. 
    Assume for a contradiction that~${|F| > 1}$ and let~${f_1, f_2 \in F}$ be distinct.
    Let~$C_x$ and~$C_y$ denote the components of~${P - F}$ containing~$x$ and~$y$ respectively. 
    Now~$x$ and~$y$ are in the same component of~${P - (F - e)}$ for every~${e \in F}$ by the minimality of~$F$. 
    Hence the end-vertices of any~${e \in F}$ meet both~$C_x$ and~$C_y$. 
    But then the end-vertices of $f_1$ are in the same component~$C$ of ${P - f_1}$, since $C$ contains~$C_x$,~$C_y$ and~$f_2$. 
    But this contradicts that the start- and end-vertex of~$P$ are separated in~${P - f_1}$. 
\end{proof}

The adapted notion of circles is analogous. 
A graph-like space is a \emph{pseudo-circle} if it is a compact connected graph-like space with at least one edge satisfying the following:
\begin{itemize}
    \item removing any edge from $C$ does not disconnect $C$ but removing any pair does;
    \item any two vertices of $C$ can be separated in $C$ by removing a pair of edges.
\end{itemize}

Pseudo-arcs and pseudo-circles are related as follows:

\begin{lemma}
    \label{lem:pseudo-circle}
    \cite{GLS}
    Let $G$ be a graph-like space, $C$ a pseudo-circle in~$G$ and ${e \in E(C)}$. 
    Then ${C-e}$ is a pseudo-arc in~$G$ joining the end-vertices of~$e$.

    Conversely, let $P$ and $Q$ be nontrivial non-loop pseudo-arcs in~$G$ that meet precisely in their end-vertices. 
    Then ${P \cup Q}$ is a pseudo-circle in~$G$.
\end{lemma}

Given two graph-like spaces $G_1$, $G_2$, a map ${\phi: G_1 \to G_2}$ is an \emph{isomorphim of graph-like spaces} if it is a homeomorphism (for the topological spaces) and it induces a bijection between~$V(G_1)$ and~$V(G_2)$.

Let~$G$ be a graph-like space and ${F \subseteq E(G)}$ a set of edges of~$G$.
We define a relation~$\sim'_F$ on~$G$ via 
\[
    \iota_e(x) \sim_F \iota_e(y) \text{ for all } e \in F \text{ and } x,y \in [0,1].
\]
Let~$\sim_F$ denote the minimal equivalence relation that extends the transitive and reflexive closure of~$\sim'_F$ such that the resulting quotient space ${G/F := G/\sim_F}$ is Hausdorff.

\begin{remark}
    The \emph{contraction} ${G/F}$ of~$F$ in~$G$ is a graph-like space with 
    vertex set ${V(G/F) := \{ [v] \in G/\sim_F \,|\, v\in V(G) \}}$, 
    edge set\footnote{This is a slight abuse of notation since technically the inner points of an edge~$e$ in the quotient space are of the form $\{(x,e)\}$ and not $(x,e)$.} 
    ${E(G/F) := E(G) \setminus F}$ 
    and for each edge ${e \in E(G)\setminus F}$ the map ${\iota_e^{G/F} := \iota_e^G}$.
\end{remark}

One can also easily show that each equivalence class with respect to~$\sim_F$ is connected in~$G$.
Moreover, we write~${G.F}$ for~${G/(E(G) \setminus F)}$ for the \emph{contraction} to~$F$ in~$G$.

We say that a graph-like space~$G'$ is a \emph{minor} of graph-like space~$G$ if there are disjoint edge sets~${F_1, F_2 \subseteq E(G)}$ and a set~${W \subseteq V(G/F_1)-F_2)}$ of non-end-vertices such that~$G'$ is isomorphic to ${((G/F_1)-F_2)-W}$.

We will also need the following fact about graph-like spaces:

\begin{theorem}
    \label{thm:GLS_connected}
    A compact graph-like space is connected if and only if it is pseudo-arc connected.
\end{theorem}

\begin{proof}
    The backward implication is clear as pseudo-arcs are connected. 
    
    For the forward implication let~${a, b \in V(G)}$ be given. 
    Consider the poset~$\mathfrak{C}$ of all closed and connected sub-graph-like spaces of~$G$ that contain both~$x$ and~$y$, ordered by inclusion. 
    Let~$\mathcal{C}$ be a decreasing chain in~$\mathfrak{C}$. 
    It is easy to verify that ${\bigcap \mathcal{C}}$ is a lower bound of~$\mathcal{C}$, where the connectedness of~${\bigcap \mathcal{C}}$ follows from a standard topology lemma that the intersection of a decreasing chain of non-empty closed connected subsets of a compact Hausdorff space are connected. 
    Hence~$\mathfrak{C}$ has a minimal element~$P$ by Zorn's Lemma. 
    Now $P$ is a pseudo-arc with start-vertex~$a$ and end-vertex~$b$, since if $P - e$ would not separate~$a$ and~$b$ then the component of~${P - e}$ containing both~$a$ and~$b$ defines a smaller compact connected sub-graph-like space than~$P$, a contradiction. 
\end{proof}

\subsection{Tree-like spaces}\label{sec:TLSproof}

There are many different equivalent ways of defining the graph-theoretical trees, which is an easy exercise to prove.

\begin{proposition}
    \label{prop:graph-trees}
    For a graph $T=(V,E) $ the following are equivalent.
    \begin{enumerate}[label=(\roman*)]
        \item For any two vertices ${a,b\in V(T)}$ there is a unique path in~$T$ from~$a$ to~$b$;
        \item $T$ is connected but ${T-e}$ is not for any edge~${e\in E(T)}$;
        \item $T$ is connected and contains no cycle.
        \item $T$ contains no cycle but every graph $T'$ with ${V(T')=V(T)}$ and $T'-F=T$ for some non-empty $F\subseteq E(T')\smallsetminus E(T)$ does.
    \end{enumerate}
\end{proposition}

A graph~$T$ is a tree if it has one (and thus all) of the above properties. 
In some situations one of these properties is easier to work with than the others, and their equivalence is used implicitly in many places in graph theory.

The above properties can be translated into the setting of graph-like spaces to say when a graph-like space is tree-like as follows:

\begin{definition}
    \label{def:TLS}
    A compact loop-free graph-like space~$G$ is a \emph{tree-like space} if one of the following conditions holds:
    \begin{enumerate}[label=\textnormal{(\roman*)}]
        \item\label{item:TLS-unique} For any two vertices ${a,b\in V(G)}$ there is a unique pseudo-arc in~$T$ from~$a$ to~$b$;
        \item\label{item:TLS-min-con} $G$ is connected but ${G-e}$ is not for any edge~${e\in E(G)}$;
        \item\label{item:TLS-con-acyclic} $G$ is connected and contains no pseudo-circle;
        \item\label{item:TLS-max-acyclic} $G$ contains no pseudo-circle but every graph-like space~$G'$ 
            with~${V(G') = V(G)}$ and~${G' - F = G}$ for some non-empty~${F \subseteq E(G') \setminus E(G)}$ does.
    \end{enumerate}
\end{definition}

Analogous to Proposition~\ref{prop:graph-trees}, we prove the following proposition.

\begin{proposition}
    \label{prop:TLS_equivalent}
    For compact loop-free graph-like spaces the conditions in Definition~\ref{def:TLS} are equivalent.
\end{proposition}

The argument is very similar to the proof of Proposition~\ref{prop:graph-trees}, but one additional technical lemma is needed: 
if two vertices~$a$ and~$b$ of a graph~$G$ are joined by two different paths it is obvious that some edge~${e\in E(G)}$ lies on exactly one of the two paths.
However for graph-like spaces and pseudo-arcs this intuitive fact requires a surprising amount of set-up to prove (see \cite{GLS}).

We forego this technical set-up and simply use the following lemma:

\begin{lemma}
    \label{lem:parc_edges}
    \cite{GLS}*{Remark~4.4}
    Any nontrivial pseudo-arc in a graph-like space is the closure of the inner points of its edges.
\end{lemma}

Lemma~\ref{lem:parc_edges} immediately implies that if two vertices~$a$ and~$b$ of a graph-like space~$G$ are joined by two distinct pseudo-arcs~$P$ and~$Q$ then there is an edge~${e \in E(G)}$ which lies on exactly one of the two pseudo-arcs. 
In fact slightly more is true: both~$P$ and~$Q$ contain an edge that does not lie on the other pseudo-arc. 
For if the edge set of~$Q$ was a proper subset of the edge set of~$P$ then~$Q$ would be disconnected as the removal of any edge from~$P$ separates~$a$ and~$b$ in~$P$.

\begin{proof}[Proof of Proposition~\ref{prop:TLS_equivalent}] \

    \ref{item:TLS-unique} $\Rightarrow$ \ref{item:TLS-max-acyclic}:
    Let~$G$ be a compact loop-free graph-like space with property~\ref{item:TLS-unique}. 
    Suppose~$C$ is a pseudo-circle in~$G$; then for any~${e \in E(C)}$ both~$e$ and~${C-e}$ define pseudo-arcs in~$G$ joining the end-vertices of~$e$, contradicting~\ref{item:TLS-unique}.
    Now let $G'$ be a graph-like space with $V(G')=V(G)$ and $G'-F=G$ for some non-empty $F\subseteq E(G')\smallsetminus E(G)$. Let $e\in F$ be an edge with end-vertices $a$ and $b$. Then $e$ defines a pseudo-arc $P$ between $a$ and $b$ in $G'$. Let $Q$ be the unique pseudo-arc in $G$ joining $a$ and $b$. Then $P$ and $Q$ intersect only in $a$ and $b$, and hence their union is a pseudo-circle in $G'$ by Lemma~\ref{lem:pseudo-circle}.
    
    \ref{item:TLS-max-acyclic} $\Rightarrow$ \ref{item:TLS-con-acyclic}:
    Let~$G$ be a compact loop-free graph-like space with property~\ref{item:TLS-max-acyclic}. 
    Suppose~$G$ is not connected. Then $G$ is not pseudo-arc connected by Theorem~\ref{thm:GLS_connected}.
    Let $a$ and $b$ be a pair of vertices that are not connected by any pseudo-arc in $G$. In particular there is no edge between $a$ and $b$. Let $G'$ be a graph-like space with $V(G')=V(G)$ such that $G=G'-\{e\}$, where $e$ is an edge in $G'$ joining $a$ and $b$. 
    Then $G'$ contains a pseudo-circle~$C$, which has to contain $e$ as otherwise $C$ would be a pseudo-circle in $G$. But then by Lemma~\ref{lem:pseudo-circle} ${C-e}\subseteq G$ is a pseudo-arc between the end-vertices of $e$, showing that $a$ and $b$ are joined by a pseudo-arc in $G$.
    
    \ref{item:TLS-con-acyclic} $\Rightarrow$ \ref{item:TLS-min-con}: 
    Let~$G$ be a compact loop-free graph-like space with property~\ref{item:TLS-con-acyclic}. 
    Suppose~${G-e}$ is still connected for some~${e \in E(G)}$ with end-vertices~$a$ and~$b$. 
    Then ${G-e}$ contains a pseudo-arc~$P$ between~$a$ and~$b$ by Theorem~\ref{thm:GLS_connected}, which together with~$e$ forms a pseudo-circle by Lemma~\ref{lem:pseudo-circle}.

    \ref{item:TLS-min-con} $\Rightarrow$ \ref{item:TLS-unique}: 
    Let~$G$ be a compact loop-free graph-like space with property~\ref{item:TLS-min-con}. 
    Theorem~\ref{thm:GLS_connected} implies that~$G$ is pseudo-arc connected. 
    For the uniqueness suppose~$G$ contains two different pseudo-arcs~$P$ and~$Q$ between two vertices~$a$ and~$b$. 
    Lemma~\ref{lem:parc_edges} implies that there is an edge~${e \in E(G)}$ which lies on exactly one of the two pseudo-arcs. 
    But then~${G-e}$ is still pseudo-arc connected\footnote{See Lemma~4.16 in~\cite{GLS}.} and therefore connected, a contradiction.
\end{proof}

\vspace{0.2cm}

Similarly to graph-theoretical trees every tree-like space gives rise to a regular tree set, see Subsection~\ref{sec:edgeTSofTLS}. 
We will show that the tree-like spaces are rich enough that one can obtain every regular tree set from them.
This is in contrast to Section~\ref{sec:tame-tree-sets} where we showed that the regular tree sets coming from trees are precisely those with no chain of order type~${\omega+1}$.
This restriction was owed to the fact that graph-theoretical trees cannot have edges that are the limit of other edges. 
But tree-like spaces \emph{can} have limit edges, so this is no longer a restriction.

In Subsection~\ref{sec:TStoTLS} we construct a corresponding regular tree set for a given tree-like space, and in Subsection~\ref{sec:treesetsandTLS} we will prove the characterisation analogously to the one in Section~\ref{sec:tame-tree-sets} by showing:
\begin{itemize}
    \item the tree-like space constructed from the edge tree set of a tree like space~$T$ is isomorphic to~$T$;
    \item the edge tree set of the tree-like space constructed from a regular tree set~$\tau$ is isomorphic to~$\tau$.
\end{itemize}

\subsection{The edge tree set of a tree-like space}
\label{sec:edgeTSofTLS}

For a tree-like space~$T$ we can define the \emph{edge tree set}~$\tau(T)$ in a way that is very similar to the definition of~$\tau(T)$ in Section~\ref{sec:tame-tree-sets}. 
Let
\[ 
    \vE(T) := \big\{ (\iota_e(0),\iota_e(1)) \,\big|\, e \in E(T) \big\} \cup \big\{ (\iota_e(1),\iota_e(0)) \,\big|\, e \in E(T) \big\}
\]
be the set of \textit{oriented edges} of~$T$. 
As tree-like spaces cannot contain loops every element of~$\vE(T)$ is a pair of two distinct vertices of~$T$. 
For vertices~${u, v \in V(T)}$ let~${P(u,v)}$ be the unique pseudo-arc in~$T$ with end-vertices~$u$ and~$v$.
Then~${\tau(T) := (\vE(T),\le,*)}$ becomes a separation system by setting~${(x,y)^* := (y,x)}$ and ${ (x,y) < (v,w)}$ for~${(x,y), (v,w) \in \vE(T)}$ with~${\{x,y\} \neq \{v,w\}}$ whenever
\[ 
    P(y,v) \subseteq P(x,v) \subseteq P(x,w).
\]
It is straightforward to check that~$\tau(T)$ is a regular tree set.

\subsection{The tree-like space of a tree set}\label{sec:TStoTLS}

Let~${\tau = (\vec E, \leq, {}^*)}$ be a regular tree set; 
we define the \emph{tree-like space corresponding to~$\tau$}, denoted~${T(\tau)}$. 
Let~${V := \vO(\tau)}$ be the set of consistent orientations and~$E$ the set of unoriented separations of~$\tau$. 
As in Section~\ref{sec:tame-tree-sets} let~$O(\vs)$ be the unique~${O \in \vO(\tau)}$ in which~$\vs$ is maximal. 
We define the tree-like space~$T(\tau)$ with vertex set~$V$ and edge set~$E$, 
that is with ground set~${V \cup \big( (0,1) \times E \big)}$. 
For this we need to define the maps~${\iota_e \colon [0,1] \to T(\tau)}$.

Fix any orientation~$O'$ of~$\tau$. 
For each~${\ve \in O'}$ let ${\iota_e \colon [0,1] \to T}$ be the map
\[
    \iota_e(x) = \left\{
    \begin{array}{cl}
        O(\ev),&\qquad x=0\\
        (x,e),&\qquad 0<x<1\\
        O(\ve),&\qquad x=1
    \end{array}
    \right..
\]
So far the definition of~$V$ and the adjacencies in~$T(\tau)$ have been analogous to the construction from Section~\ref{sec:tame-tree-sets}. 
But to make~$T(\tau)$ into a graph-like space we also need to define a topology.

For~${\ve \in O'}$ let~${E^+(\ve)}$ be the set of all~${\vs \in O'}$ with~${\ve < \vs}$ or~${\ve < \sv}$, and~${E^-(\ve)}$ the set of all~${\vs \in O'}$ with~${\vs < \ve}$. 
For~${\ve \in O'}$ and~${r \in (0,1)}$ set
\[
    S(\ve,r) := \{ O \in \mathcal{O}(\tau) \, | \, \ve \in O \} \cup \big( (0,1)\times E^+(\ve) \big) \cup \big( (r,1) \times e \big) 
\]
and
\[
    S(\ev,r) := \{ O \in \mathcal{O}(\tau) \, | \, \ev \in O \} \cup \big( (0,1)\times E^-(\ve) \big) \cup \big( (0,r) \times e \big). 
\]
We define the sub-base of the topology on~$T(\tau)$ as 
$\mathcal{S} := \big\{ S(\ve,r) \,\big|\, \ve \in \tau,\, r \in (0,1)\, \big\}$. 
Note that only the notation depends on the choice of~$O'$ but the topology on~$T(\tau)$ does not. 
It is clear that~$T(\tau)$ is a graph-like space: 
for any two vertices~${a,b \in V}$ pick any~$\ve$ in the symmetric difference of~$a$ and~$b$, viewed as orientations of~$\tau$. 
Then~${S(\ve,\frac{1}{2})}$ and~${S(\ev,\frac{1}{2})}$ are disjoint open sets partitioning~$V$ and~${\{a,b\}}$.

\begin{lemma}
    \label{lem:TStoTLS-compact}
    $T(\tau)$ is compact.
\end{lemma}

\begin{proof}
    By the Alexander~sub-base~theorem from general topology it suffices to show that any open covering of sets in~$\mathcal{S}$ has a finite sub-cover. 
    Suppose that~$\mathcal{C}$ is a sub-basic open cover of~${T(\tau)}$ with no finite sub-cover. 
    Let~$E(\mathcal{C})$ be the set of all~${\ve \in \tau}$ such that~${S(\ve,x) \in \mathcal{C}}$ for some~${x \in(0,1)}$. 
    If~${\vr \leq \sv}$ for any~${\vr, \vs \in E(\mathcal{C})}$ then their corresponding sets in~$\mathcal{C}$ already cover all of~$T(\tau)$, except possibly for~${(0,1 )\times r}$ if~${\vr = \sv}$, which can be finitely covered. 
    Thus we may assume that~${\vr \not\le \sv}$ for all~${\vr, \vs \in E(\mathcal{C})}$. 
    Then the set
    \[
        E^*(\mathcal{C}) := \{ \ev \,|\, \ve \in E(\mathcal{C}) \}
    \]
    is a consistent partial orientation of~$\tau$, so by the Extension Lemma~\ref{lem:extension} there is an~${O \in \vO(\tau)}$ with~${E^*(\mathcal{C})\subseteq O}$. 
    But~${O \notin S(\ve,r)}$ for every~${\ve \in E(\mathcal{C})}$ and~${r \in(0,1)}$, so~$\mathcal{C}$ was not a cover of~$T$. 
    Therefore~$T$ is a compact graph-like space.
\end{proof}

\begin{lemma}
    \label{lem:TStoTLS-connected}
    $T(\tau)$ is connected, but~${T(\tau)-e}$ is not for every~${e \in E}$.
\end{lemma}

\begin{proof}
    The latter follows immediately from the definition of~$\mathcal{S}$: 
    for any edge~${e \in E}$ the sets~${S(\ve,\frac{1}{2})}$ and~${S(\ev,\frac{1}{2})}$ define a partition of~${T(\tau)-e}$ into non-empty disjoint open sets. 
    
    To show that~$T$ is connected first note that any non-empty open set in~$T$ contains an inner point of an edge. 
    Suppose that~${A, B}$ are non-empty disjoint open sets partitioning~$T$. 
    For any edge~${e \in E}$ the image of~$\iota_e$ in~$T$ is connected, 
    hence every edge whose inner points meet~$A$ is completely contained in~$A$, and similarly for~$B$. 
    Write~$\tau_A$ for the set of~${\ve \in \tau}$ with~${\mathring{e} \subseteq A}$, and~${\tau_B}$ for the set of~${\ve \in \tau}$ with~${\mathring{e} \subseteq B}$. 
    Then~$\tau_A$ and~$\tau_B$ partition~$\tau$ and are closed under involution. 
    Fix any~${\va \in \tau_A}$ and~${\vb \in \tau_B}$ with~${\va \leq \bv}$ 
    and write~${C := \{ \vr \in \tau \,|\, \va \leq \vr \leq \bv \}}$ for the chain of elements between~$\va$ and~$\bv$. 
    Let~$C_A$ be a maximal initial segment of~$C$ with~${C_A \subset \tau_A}$ and~${C_B}$ a maximal initial segment of~$C^*$ with~${C_B \subseteq \tau_B}$, where~$C^*$ is the image of~$C$ under the involution.
    The set~${C_A \cup C_B}$ is a consistent partial orientation of~$\tau$, so by the Extension Lemma~\ref{lem:extension} there is an~${O \in V}$ with~${C_A \cup C_B \subseteq O}$.
    Suppose that~${O \in A}$, say. 
    Let~${X \subseteq \tau}$ be minimal in size with the property that
    \[
        O \in \mathcal{X} := \bigcap_{\vx\in X} S(\vx,r(\vx)) \subseteq A
    \]
    for suitable~${r(\vx) \in (0,1)}$. 
    From our assumptions it follows that such an~$X$ exists and is a finite subset of~$O$, and the minimality implies that~$X$ is a star. 
    Observe that~${\mathring{b} \subseteq S(\vx,r(\vx))}$ for all~${\vx \in X}$ with~${\vx < \bv}$. 
    As~$\mathcal{X}$ does not meet~$B$ there must be a (unique) ${\vx \in X}$ with~${\vx \geq \vb}$ and thus~${\xv \in C}$. 
    If~${\vx \in \tau_B}$ then~$\mathcal{X}$ again meets~$B$, hence~${\vx \in \tau_A}$. 
    As ${\vx \in O}$ and thus ${\xv \notin C_A}$, there is a ${\vt \in \tau_B \cap O}$ with~${\vx \leq \vt}$. 
    But then~${\mathring{t}\subset\mathcal{X}}$, a contradiction. 
    Therefore~${T(\tau)}$ is connected.
\end{proof}

Hence we have shown that $T(\tau)$ is indeed a tree-like space.

\subsection{Regular tree sets and tree-like spaces -- A characterisation}
\label{sec:treesetsandTLS}

\begin{lemma}
    \label{lem:tau-T-tau-iso-TLS}
    Any regular tree set $\tau'$ is isomorphic to $\tau(T(\tau'))$.
\end{lemma}

\begin{proof}
    For two vertices ${u, v \in \vO(\tau')}$ the set ${C = v \setminus u}$ is a chain in~$\tau'$. 
    Set
    \[ 
        P(u,v) := \overline{ \bigcup \{ \mathring{e} \,|\, \ve \in C \}} \subseteq T(\tau').
    \]
    Then ${P(u,v) = P(v,u)}$ and~${P(u,v)}$ is the unique pseudo-arc in~$T$ with~$u$ and~$v$ as end-vertices\footnote{This follows immediately if one uses the machinery established in \cite{GLS}, which we do not introduce here. 
    Alternatively one can show the connectedness of~$P(u,v)$ by repeating the proof that~$T(\tau')$ is connected, and verifying the other properties of a pseudo-arc directly.}. 
    Define the map~${\phi \colon \tau' \to \vE(T(\tau'))}$ as
    \[
        \phi(\ve) := 
        \left\{
        \begin{array}{cl}
            (\iota_e(0),\iota_e(1)), \qquad \ve \in \iota_e(1)\\
            (\iota_e(1),\iota_e(0)), \qquad \ve \in \iota_e(0)
        \end{array}
        \right..
    \]
    This is a bijection between~$\tau'$ and~${\vE(T(\tau'))}$ that commutes with the involution. 
    The claim follows from Lemma~\ref{lem:Isomorphism} if we can show that~$\phi$ is order-preserving. 
    For this let~${\vr, \vs \in \tau'}$ with~${\vr < \vs}$. 
    Let~${(x,y)}$ be the end-vertices of~${r \in E(T(\tau'))}$ with~${\vr \in y}$ and~${(v,w)}$ the end-vertices of ${s \in E(T(\tau'))}$ with~${\vs \in w}$. 
    Then
    \[
        v \setminus y = (v\setminus x) \setminus \{\vr\}
    \]
    and
    \[
        v \setminus x = (w\setminus x) \setminus \{\vs\},
    \]
    so~${P(y,v) \subseteq P(x,v)\subseteq P(x,w)}$ and hence~${\phi(\vr) = (x,y) \leq (v,w) = \phi(\vs)}$.
\end{proof}

\begin{lemma}
    \label{lem:T-tau-T-iso-TLS}
    Any tree-like space $T'$ is isomorphic to $T(\tau(T'))$.
\end{lemma}

\begin{proof}
    For ease of notation, we may assume without loss of generality that the arbitrary orientation of~$\tau(T')$ we fixed for the construction of~${T(\tau(T'))}$ is ${\{ ( \iota^{T'}_e(0), \iota^{T'}_e(1) )  \, | \, e \in E(T') \}}$.
    
    For every edge~${e \in E(T')}$ there is a unique~${j(v,e) \in \{0,1\}}$ such that~$v$ is in the same component of~${T' - e}$ as~${\iota^{T'}_e(j(v,e))}$ by Proposition~\ref{prop:TLS_equivalent}.
    We define a map ${\phi: V(T') \to V(T(\tau(T')))}$ by setting~$\phi(v)$ to be the orientation 
    \[
        \{ ( \iota^{T'}_e(1-j(v,e)), \iota^{T'}_e(j(v,e)) ) \, | \, e \in E(T) \}
    \]
    of~$\tau(T')$, which is easily verified to be consistent.
    
    We extend~$\phi$ to a map~${T' \to T(\tau(T'))}$ by setting ${\phi(r,e) := (r,\{ \iota^{T'}_e(0), \iota^{T'}_e(1) \})}$  
    for~${r \in (0,1)}$ and~${e \in E(T')}$.
    It is easy to check that~$\phi$ is a bijection and induces a bijection between~${V(T')}$ and~${V(T(\tau(T')))}$.
    Since~$T'$ is compact and~${T(\tau(T'))}$ is Hausdorff, we only need to check that~$\phi$ is continuous.
    For each~${e \in E(T')}$ and each~${r \in (0,1)}$ note that~${T' \setminus \{r\}}$ contains two connected components~${C(e,r,0)}$ and~${C(e,r,1)}$, where~${C(e,r,j)}$ denotes the component containing~${\iota^{T'}_e(j)}$.
    By construction, ${\phi(C(e,r,j)) = S( (\iota^{T'}_e(1-j),\iota^{T'}_e(j)), r)}$ and hence the preimage of any subbasis element is open.
\end{proof}

Altogether we have proven the main theorem of this section.

\begin{theorem}
    \begin{enumerate}
        \item A tree set is isomorphic to the edge tree set of a tree-like space if and only if it is regular.
        \item Any regular tree set~$\tau'$ is isomorphic to~${\tau(T(\tau'))}$.
        \item Any tree-like space~$T'$ is isomorphic to~${T(\tau(T'))}$.\qed
    \end{enumerate}
\end{theorem}

Additionally, for distinct but comparable tree sets, we can say precisely in which way the corresponding trees from Theorem~\ref{thm:tametreesets} above are comparable: one will be a minor of the other.

Let us finish this section with two further results on how these constructions relate to substructures.

\begin{theorem}\label{thm:tlsminor1}
    Let~$\tau_1$,~$\tau_2$ be regular tree-sets with~${\tau_1 \subseteq \tau_2}$. 
    Then~${T(\tau_1)}$ is a minor of~${T(\tau_2)}$.
\end{theorem}

\begin{proof}
    We show that~${T_1 := T(\tau_1)}$ is isomorphic to~${T_2 := T(\tau_2).E(T(\tau_1))}$.
    
    First we note that~${\mathcal{O}(\tau_1) = \{ O \cap \tau_1 \,|\, O \in \mathcal{O}(\tau_2) \}}$.
    Moreover it immediately follows from the definitions that~${O, O' \in \mathcal{O}(\tau_2)}$ are representatives of the same vertex of~$T_2$ if and only if~${O \cap \tau_1 = O' \cap \tau_1}$.
    
    For ease of notation we may assume without loss of generality that the orientation of~$\tau_1$ that we chose in the construction of~${T(\tau_1)}$ is induced by the orientation we chose for~$\tau_2$ in the construction of~${T(\tau_2)}$.
    Let~$\phi$ denote the concatenation of the identity from~$T_1$ to~${T(\tau_2)}$ and the quotient map from~${T(\tau_2)}$ to~$T_2$.
    By the previous observations, this map is a bijection and induces a bijection between~${V(T_1)}$ and~${V(T_2)}$.
    By definition~$\phi$ is continuous and hence shows that~$T_1$ is isomorphic to~$T_2$.
\end{proof}
    
\begin{theorem}\label{thm:tlsminor2}
    Let~$T_1$,~$T_2$ be tree-like spaces where~$T_1$ is a minor of~$T_2$.
    Then~${\tau(T_1)}$ is isomorphic to a subset of~${\tau(T_2)}$.
\end{theorem}

\begin{proof}
    For ease of notation we may assume without loss of generality that~${T_1 = T_2.E(T_1)}$ and that ${\iota^{T_2}_e(j) \in \iota^{T_1}_e(j)}$ for all ${e \in E(T_1)}$ and~${j \in \{0,1\}}$.
    We show that ${\tau_1 := \tau(T_1)}$ is isomorphic to ${\tau_2 := \tau(T_2) \setminus \{ (v,w) \,|\, v \in [w] \}}$.
    
    Let~${\phi: \tau_2 \to \tau_1}$ be defined as~${\phi(v,w) = ([v], [w])}$.
    It is easy to see that this map is well-defined, surjective and commutes with the involution.
    For the injectivity consider ${(v_1,w_1), (v_2,w_2) \in \tau_2}$ with~${v_1 \in [v_2]}$ and~${w_1 \in [w_2]}$ and let~${e_i \in E(T_2)}$ be such that ${\{v_i, w_i\} = \{ \iota^{T_2}_{e_i}(0), \iota^{T_2}_{e_i}(1) \}}$ for~${i \in \{1,2\}}$.
    Since~${[v_2]}$ and~${[w_2]}$ are both connected (as subspaces of~$T_2$) but in different components of~${T_2 - e_i}$, we obtain that~${e_1 = e_2}$ and hence ${(v_1,w_1) = (v_2,w_2)}$.
    
    Consider a pseudo-arc~${P(v,w)}$ in~$T_2$ between any vertices~$v$ and~$w$.
    It is not hard to verify that the unique pseudo-arc in~$T_1$ between~${[v]}$ and~${[w]}$ has as its point set ${\{ [x] \in T_1 \,|\, x \in P(v,w) \}}$.
    This observation implies that~$\phi$ is order-preserving and hence an isomorphism by Lemma~\ref{lem:Isomorphism}.
\end{proof}

\begin{lemma}
    Let~$P$ be a pseudo-arc and~${x, y \in V(P)}$. 
    Then there is an edge~${e \in E(P)}$ such that~$x$ and~$y$ are separated in~${P-e}$. 
\end{lemma}

\section*{Acknowledgement}

We would like to thank Nathan Bowler for greatly simplifying our proof of Theorem~\ref{thm:GLS_connected} by pointing out the redundancy of the extra condition in the definition of pseudo-arc (cf.~Lemma~\ref{lem:pseudo-arc-redundancy}).

\end{document}